\documentclass[12pt]{article}

\usepackage{amsmath, amsthm, amssymb}
\usepackage{multirow}
\usepackage{multicol}
\usepackage{fullpage}
\usepackage{booktabs}
\usepackage{colortbl}
\usepackage{graphicx}
\usepackage{array}
\usepackage{setspace}

\theoremstyle{plain}
\newtheorem{theorem}{Theorem}
\newtheorem{lemma}[theorem]{Lemma}
\newtheorem{draft-claim}[theorem]{Claim (draft - do not trust)}

\newtheorem{corollary}[theorem]{Corollary}

\newtheorem{observation}[theorem]{Observation}

\newtheorem{question}[theorem]{Question}

\theoremstyle{definition}
\newtheorem{definition}[theorem]{Definition}

\newcommand{\sqn}[1]{\langle #1 \rangle}   

\newcommand{\define}[1]{\emph{#1}}

\newcommand{\bitem}{\begin{itemize}}
\newcommand{\eitem}{\end{itemize}}
\newcommand{\benum}{\begin{enumerate}}
\newcommand{\eenum}{\end{enumerate}}

\newcommand{\ZZ}{\mathbb{Z}}

\newcommand{\SQ}{\mathcal{Q}}

\newcommand{\normal}{\unlhd}

\usepackage{amsmath, amsthm, amssymb}
\newcommand{\PI}{P^\omega}
\newcommand{\LMlt}{\mbox{LMlt}}
\newcommand{\Mlt}{\mbox{Mlt}}
\newcommand{\hcell}[1]{[#1]}

\begin{document}

\title{Products of all elements in a loop and a framework for non-associative analogues of the Hall-Paige conjecture}
\author{Kyle Pula \\
\small Department of Mathematics\\[-0.8ex]
\small University of Denver, Denver, CO, USA\\
\small \texttt{jpula@math.du.edu}
}

\date{
July 11, 2008
}

\maketitle

\begin{abstract}
For a finite loop $Q$, let $P (Q)$ be the set of elements that can be represented as a product containing each element of $Q$ precisely once. Motivated by the recent proof of the Hall-Paige conjecture, we prove several universal implications between the following conditions:
\begin{enumerate}
\item[(A)] $Q$ has a complete mapping, i.e. the multiplication table of $Q$ has a transversal,
\item[(B)] there is no $N \normal Q$ such that $|N|$ is odd and $Q/N \cong \ZZ_{2^m}$ for $m \geq 1$, and
\item[(C)] $P(Q)$ intersects the associator subloop of $Q$.
\end{enumerate}
We prove $(A) \implies (C)$ and $(B) \iff (C)$ and show that when $Q$ is a group, these conditions reduce to familiar statements related to the Hall-Paige conjecture (which essentially says that in groups $(B) \implies (A))$. We also establish properties of $P(Q)$, prove a generalization of the D\'enes-Hermann theorem, and present an elementary proof of a weak form of the Hall-Paige conjecture.
\end{abstract}

\section{Background}

As the present work is motivated by the search for non-associative analogues of the well-known Hall-Paige conjecture in group theory, we begin with a brief historical sketch of this topic. As the culmination of over 50 years of work by many mathematicians, the following theorem seems now to have been established.
\begin{theorem}\label{THM:HP}
If $G$ is a finite group, then the following are equivalent:
\begin{enumerate}
\item[$(1)$] $G$ has a complete mapping, i.e. the multiplication table of $G$ has a transversal,
\item[$(2)$] Sylow $2$-subgroups of $G$ are trivial or non-cyclic, and
\item[$(3)$] there is an ordering of the elements of $G$ such that $g_1 \cdots g_n = 1$.
\end{enumerate}
\end{theorem}

The history of Theorem \ref{THM:HP} dates back at least to a result of Paige from 1951 in which he proves (1)~$\implies$~(3) \cite{Paige:1951hb}. In 1955 Hall and Paige proved that for all finite groups (1)~$\implies$~(2) and that for solvable, symmetric, and alternating groups the converse holds as well. They conjectured that the converse holds in all finite groups, and this claim came to be known as the Hall-Paige conjecture \cite{Hall:1955zr}.

In 1989 D\`enes and Keedwell noted that condition (2) holds in all non-solvable groups and proved that condition (3) does as well \cite{Denes:1989qp}. Combining the above results, they observed that the Hall-Paige conjecture was thus equivalent to the statement that non-solvable groups have complete mappings, and many authors have since contributed to the effort of demonstrating such complete mappings. Evans provides an excellent survey of this progress up to 1992 in \cite{Evans:1992zl} and of the more recent progress in \cite{Evans:2008gy}. Jumping ahead to the most recent developments (consult the previous references for complete details), Wilcox improved on earlier results to show that a minimal counterexample to the Hall-Paige conjecture must be simple and reduced the number of candidates to the Tits group or a sporadic simple group (some of which were already known to have complete mappings) \cite{Wilcox:2007qf}. The closing work on these efforts was completed by Bray and Evans who demonstrated complete mappings in $J_4$ and all remaining cases, respectively \cite{Bray:2008yg, Evans:2008gy}.

We would like to consider the Hall-Paige conjecture in more general varieties of loops, and this paper takes on the modest goal of making sense of conditions (1), (2), and (3) of Theorem \ref{THM:HP} in non-associative settings. While (1) translates directly, (2) and (3) present difficulties since it is not clear what a Sylow 2-subloop should be and any product in a non-associative loop requires the specification of some association.

The primary contribution of this paper is to propose natural generalizations of these conditions and establish several universal implications between them. As detailed in \S \ref{SEC:HP-condition}, we prove a number of related results including a generalization of the D\'enes-Hermann theorem and provide an elementary proof of a weak form of the Hall-Paige conjecture.

We begin with a minimal background on the theories of latin squares and loops.
\subsection{Latin squares}

A \define{latin square of order $n$} is an $n \times n$ array whose entries consist of $n$ distinct symbols such that no symbol appears as an entry twice in a row or column. More formally, a Latin square $L$ on a finite set of symbols $S$ is a subset of $S^3$ such that the projection along any pair of coordinates is a bijection. When the triple $(x,y,z) \in L$, we say that the symbol $x$ appears as a row, $y$ as a column, and $z$ as an entry.

A \define{$k$-plex} is a subset of $L$ in which each symbol appears as a row, column, and entry precisely $k$ times. A $1$-plex is called a \define{transversal} and a $k$-plex with $k$ odd is called an odd-plex. For the theory and history of $k$-plexes, consult \cite{Bryant:2009zl, Egan:2008eu, Wanless:2002rr, Wanless:2007zr}.

We introduce the following generalizations of the concepts of transversals and $k$-plexes. A \define{row $k$-plex} of $L$ is a collection of triples in $L$ representing each row precisely $k$ times. We call a subset $C = \{(x_i, y_i, z_i) : 1 \leq i \leq m \} \subseteq L$ \define{column-entry regular}, or just \define{regular} for short, if for each symbol $s$ we have $|\{i : y_i = s\}| = |\{i : z_i = s\}|$. That is, $s$ appears as an entry the same number of times it appears as a column. We denote by $C_r$ the multiset of symbols appearing as rows in $C$. For example, if $C$ is a $k$-plex, then $C_r$ contains precisely $k$ copies of each symbol. We will be primarily interested in regular row transversals, i.e. selections of a single cell from each row so that each symbol appears as a column the same number of times as an entry (Figure \ref{FIG:HP-counter-example} depicts such a selection in the multiplication table of a loop).

\subsection{Loops}

A set with a binary operation, say $(Q, \cdot)$, is a \define{loop} if for each $x,z \in Q$, the equations $x \cdot y_1 = z$ and $y_2 \cdot x=z$ have unique solutions $y_1, y_2 \in Q$ and $(Q, \cdot)$ has a neutral element (which we always denote $1$). A \define{group} is a loop in which the associativity law holds. We assume in all cases  that $Q$ is finite and typically write $Q$ rather than $(Q, \cdot)$.

The \define{left}, \define{right}, and \define{middle inner mappings} of a loop are defined as $L(x, y) = L^{-1 }_{yx} L_y L_x$, $R(x, y) = R^{-1}_{xy} R_y R_x$, and $T(x) = R^{-1}_x L_x$, respectively. A subloop $S$ of a loop $Q$ is said to be \define{normal}, written $S \normal Q$, if $S$ is invariant under all inner mappings of $Q$. A loop $Q$ is \define{simple} if it has no normal subloops except for $\{1\}$ and $Q$. This definition of normality is equivalent to what one would expect from the standard universal algebraic definition. In fact, just as with groups, loops can be defined more formally as universal algebras though we have not elected to do so here.

We write $A(Q)$ for the associator subloop of $Q$, the smallest normal subloop of $Q$ such that $Q/A(Q)$ is a group. Likewise, we write $Q'$ for the derived subloop of $Q$, the smallest normal subloop of $Q$ such that $Q/Q'$ is an Abelian group. Note that $A(Q) \normal Q'$ and thus cosets of $Q'$ are partitioned by cosets  of $A(Q)$. When $Q$ is a group, $A(Q) = \{1\}$.

The \define{multiplication table} of a loop $Q$ is the set of triples $\{(x, y, xy) : x,y \in Q\}$. Multiplication tables of loops are latin squares and, up to the reordering of rows and columns, every latin square is the multiplication table of some loop. It thus makes sense to say that a loop has a $k$-plex or more generally a regular row $k$-plex whenever its multiplication table does.

Most literature on the Hall-Paige conjecture focuses on the concept of a complete mapping of a group rather than a transversal of its multiplication table, though the two are completely equivalent \cite[p. 7]{Denes:1991lq}. In the general loop setting, we prefer the latter concept as it emphasizes the combinatorial nature of the problem and generalizes more naturally to the concepts of $k$-plexes and regular $k$-plexes.

For $H \subseteq Q$ and $k \geq 1$, let $P^k(H)$ be the set of elements in $Q$ that admit factorizations containing every element of $H$ precisely $k$ times. We call these elements \define{full $k$-products of $H$}. When $k=1$, we write just $P(H)$ and refer to its elements as \define{full products of $H$}. We work primarily in the case $H = Q$ and simply refer to these elements as \define{full $k$-products}. While in the group case, the set $P(G)$ has been well-studied (see the commentary preceding Theorem \ref{THM:denes-hermann} for some background), to the best of our knowledge, the present article contains the first investigation of the general loop case.

Although we assume a basic familiarity with both loops and latin squares, we provide references for any non-trivial results that we employ. For standard references on latin squares consult D\'enes and Keedwell \cite{Denes:1974ul, Denes:1991lq} and for loops Bruck \cite{Bruck:1958ly} and Pflugfelder \cite{Pflugfelder:1990bh}.

\section{Summary of Results}\label{SEC:HP-condition}

We propose the following conditions as fruitful interpretations of $(1)$, $(2)$, and $(3)$ of Theorem \ref{THM:HP} in varieties of loops in which associativity need not hold.

\begin{definition}[HP-condition]
We say a class of loops $\SQ$ satisfies the \define{HP-condition} if for each $Q \in \SQ$ the following are equivalent:
\begin{enumerate}
\item[(A)] $Q$ has a transversal,
\item[(B)] there is no $N \normal Q$ such that $|N|$ is odd and $Q/N \cong \ZZ_{2^m}$ for $m \geq 1$, and
\item[(C)] $A(Q)$ intersects $P(Q)$.
\end{enumerate}
\end{definition}
When $\SQ$ is the variety of groups, satisfaction of the HP-condition reduces to Theorem \ref{THM:HP}. The equivalence of (1) and (A) is clear; as is that of (3) and (C), given that when $Q$ is a group, $A(Q) = \{1\}$. We take an indirect approach to showing (2) $\iff$ (B) by showing (B) $\iff$ (C), a corollary of Theorems \ref{THM:DH-loops} and \ref{THM:pqnot-qsteptype}. In 2003, Vaughan-Lee and Wanless gave the first elementary proof of (2) $\iff$ (3). Their paper also provides some background on this result (whose initial proof invoked the Feit-Thompson theorem) \cite{Vaughan-Lee:2003cr}. As corollaries of our main results, we show that in all loops (B)~$\iff$~(C) and (A)~$\implies$~(C).

The following easy observation sets the context for our main results. It is well-known in the group case and follows in the loop case for the same simple reasons. We include it as part of Lemma \ref{LEM:pq-properties} for completeness.

\begin{observation} $P^k(Q)$ is contained in a single coset of $Q'$.
\end{observation}

At least as far back as 1951, authors have asked whether, in the group case, this observation can be extended to show that $P^k(Q)$ in fact coincides with this coset. For a history of this line of investigation, see \cite[p. 35]{Denes:1974ul} and \cite[p. 40]{Denes:1991lq}. This result now bears the names of D\'enes and Hermann who first established the claim for all groups.

\begin{theorem}[D\'enes, Hermann \cite{Denes:1982qf} 1982]\label{THM:denes-hermann}
If $G$ is a group, then $P(G)$ is a coset of $G'$. It follows that $P^k(G)$ is also a coset of $G'$.
\end{theorem}

An admittedly cumbersome but more general way to read this statement is that $P(G)$ intersects every coset of $A(G)$ that is contained in the relevant coset of $G'$. Since $A(G) = \{1\}$ and these cosets partition cosets of $G'$, this technical phrasing reduces to the theorem as stated. We extend the D\'enes-Hermann theorem to show that this more general phrasing holds in all loops. Although the result is more general, our proof of Theorem \ref{THM:DH-loops} relies essentially upon the D\'enes-Hermann theorem. In \S \ref{SEC:DenesHermann}, we discuss prospects of a further generalization.

\begin{theorem}\label{THM:DH-loops} If $P(Q) \subseteq xQ'$, then
$P(Q) \cap y A(Q) \neq \emptyset$ for all $y \in xQ'$. That is, $P(Q)$ intersects every coset of $A(Q)$ contained in $xQ'$ and, in particular, if $P(Q) \subseteq Q'$, then $P(Q)$ intersects $A(Q)$. It follows that $P^k(Q)$ also intersects every coset of $A(Q)$ in the corresponding coset of $Q'$.
\end{theorem}

Coupled with Theorem \ref{THM:DH-loops}, our next result establishes (B) $\iff$ (C).

\begin{theorem}\label{THM:pqnot-qsteptype}
$P(Q) \subseteq Q'$ if and only if $(B)$ holds.
\end{theorem}

In 1951, Paige showed that if a group $G$ has a transversal, then $1 \in P(G)$ \cite{Paige:1951hb}. We extend this result to a much wider class of structures.

\begin{theorem}\label{THM:kplex-kprod}
If $C$ is a regular subset of the multiplication table of $Q$, then $P(C_r)$ intersects $A(Q)$.
In particular, if $Q$ has a $k$-plex (or just a regular row $k$-plex), then $P^k(Q)$ intersects $A(Q)$.
\end{theorem}

Applying these results, we establish that for all loops:
\begin{itemize}
\item (A) $\implies$ (C) by Theorem \ref{THM:kplex-kprod} and
\item (B) $\iff$ (C) by Theorems \ref{THM:DH-loops} and \ref{THM:pqnot-qsteptype}.
\end{itemize}

By 1779, Euler had shown that a cyclic group of even order has no transversal and in 1894 Maillet extended his argument to show that all loops for which condition (B) fails lack transversals \cite[p. 445]{Denes:1974ul}. In 2002, Wanless showed that such loops lack not just transversals, i.e. $1$-plexes, but contain no odd-plexes at all \cite{Wanless:2002rr}. While their arguments are quite nice, our proof of (A) $\Longrightarrow$ (B) provides an alternative, more algebraic proof of these results.

\begin{corollary}
If a loop fails to satisfy $(B)$, then it has no regular row odd-plexes.
\end{corollary}

It is not true in general that (B) $\wedge$ (C) $\implies$ (A) (for a smallest possible counter-example, see Figure \ref{FIG:HP-counter-example}). In a separate paper still in preparation, we show that this implication holds in several technical varieties of loops that include non-associative members and provide both computational and theoretical evidence suggesting that it may hold in the well-known variety of Moufang loops as well \cite{Pula:2008xw}.

\begin{figure}[h]
\begin{center}
\begin{tabular}{c|c c c c c c } 
$Q$ & 1 & 2 & 3 & 4 & 5 & 6 \\ \hline 
1 & \hcell{1} & 2 & 3 & 4 & 5 & 6 \\ 
2 & 2 & 1 & 4 & \hcell{3} & 6 & 5 \\  
3 & 3 & 5 & \hcell{1} & 6 & 2 & 4 \\  
4 & \hcell{4} & 6 & 2 & 5 & 1 & 3 \\  
5 & 5 & 3 & 6 & 2 & 4 & \hcell{1} \\  
6 & \hcell{6} & 4 & 5 & 1 & 3 & 2 \\  
\end{tabular}
\caption{Loop $Q$ with no transversal and yet $P(Q) = Q' = Q$. $Q$ contains $168$ regular row transversals, one of which has been bracketed.}
\label{FIG:HP-counter-example}
\end{center}
\end{figure}

The equivalent statements in Theorem \ref{THM:HP} are typically stated with the additional claim that $G$ can be partitioned into $n$ mutually disjoint transversals, i.e. $G$ has an orthogonal mate. In the group case, it is easy to show that having an orthogonal mate is equivalent to having at least one transversal. While this equivalence may extend to other varieties of loops (this question for Moufang loops is addressed directly in \cite{Pula:2008xw}), the argument seems unrelated to the most difficult part of Theorem \ref{THM:HP} (that (2) $\implies$ (1)).

We do however introduce a weakening of the orthogonal mate condition in the following theorem. While this result follows directly from a combination of the Theorems \ref{THM:HP} and \ref{THM:kplex-kprod}, we provide an elementary proof (in particular, we avoid the classification of finite simple groups).

\begin{theorem}\label{THM:one-in-pg-iff-ceregrowtrans}
If $G$ is a group of order $n$, then the following are equivalent:
\begin{enumerate}
\item $G$ has a regular row transversal,
\item $G$ can be partitioned into $n$ mutually disjoint regular row transversals,
\item Sylow $2$-subgroups of $G$ are trivial or non-cyclic, and
\item $1 \in P(G)$.
\end{enumerate}
\end{theorem}

\begin{corollary}\label{COR:HP-equiv-reg-row-trans}
Theorem \ref{THM:HP} is equivalent to the claim that a group has a transversal if and only if it has a regular row transversal.
\end{corollary}

We make the following two observations not to suggest that our methods may be useful in tackling these important problems but rather to indicate their theoretical context.

\begin{observation}
When $\SQ$ is the class of odd ordered loops, condition $(B)$ always holds and thus the implication (B)~$\implies$~(A) is equivalent to Ryser's conjecture, that every latin square of odd order has a transversal $\cite{Colbourn:1996cj}$.

A natural question might be whether $Q$ has a $2$-plex if and only if $A(Q)$ intersects $P^2(Q)$. Since this latter condition is satisfied in all loops, an affirmative answer to this question is equivalent a conjecture attributed to Rodney, that every latin square has a $2$-plex $\cite[p. 105]{Colbourn:1996cj}$.
\end{observation}

\section{Properties of the sets $P^k(Q)$}\label{SEC:pkq-properties}

We begin with a sequence of easy observations about the sets $P^k (Q)$.

\begin{lemma}\label{LEM:pq-properties}
For $i,j,k \geq 1$,
\begin{enumerate}
\item\label{LEM:pq-prop-10}
$1 \in P^2 (Q)$,
\item\label{LEM:pq-prop-20}
$P^i (Q) P^j(Q) \subseteq P^{i+j}(Q)$ and $|P^k(Q)| \leq |P^{k+1}(Q)|$,
\item\label{LEM:pq-prop-30}
$P^k(Q)$ is contained in a coset of $Q'$,
\item\label{LEM:pq-prop-40}
$P^{k}(Q) \subseteq P^{k+2}(Q)$,
\item\label{LEM:pq-prop-50}
$P^2(Q) \subseteq Q'$, and
\item\label{LEM:pq-prop-60} 
$P(Q) \subseteq aQ'$ where $a^2 \in Q'$.
\end{enumerate}
\end{lemma}
\begin{proof}
\ref{LEM:pq-prop-10}
Let $q^\rho$ be the right inverse of $q$. Then $1 = \prod_{q \in Q} q q^\rho \in P^2(Q)$.

\ref{LEM:pq-prop-20}
Observe that $P^i (Q) P^j(Q) = \{ab : a \in P^i(Q), b \in P^j(Q)\}$. Thus $ab$ is a full $(i+j)$-product. It then follows that $|P^k(Q)| \leq |P^{k+1}(Q)|$ since for $q \in P(Q)$, $q P^k(Q) \subseteq P^{k+1}(Q)$ and $|P^k(Q)| = |q P^k(Q)|$.

\ref{LEM:pq-prop-30}
Any two elements of $P^k(Q)$ have factors that differ only in their order and association. In other words, if $x,y \in P^k(Q)$, then $xQ' = yQ'$. 

\ref{LEM:pq-prop-40}
By \ref{LEM:pq-prop-10}, we have $1 \in P^2(Q)$; thus $P^k(Q)  = P^k(Q) \cdot 1 \subseteq P^k(Q) P^2(Q)$. By \ref{LEM:pq-prop-20} we have $P^k (Q) P^2 (Q) \subseteq P^{k+2} (Q)$. Thus $P^k(Q) \subseteq P^{k+2} (Q)$.

\ref{LEM:pq-prop-50}
The claim follows immediately from \ref{LEM:pq-prop-10} and \ref{LEM:pq-prop-30}.

\ref{LEM:pq-prop-60}
By \ref{LEM:pq-prop-30}, $P(Q) \subseteq aQ'$ for some $a \in Q$ and thus $P(Q)^2 \subseteq a^2 Q'$. By \ref{LEM:pq-prop-20}, $P(Q)^2 \subseteq P^2(Q)$ and by \ref{LEM:pq-prop-50} $P^2 (Q) \subseteq Q'$. It follows that $a^2 Q' = Q'$ and thus $a \in Q'$.
\end{proof}

Our next lemma uses the idea that $Q/N$ is a set of cosets of $N$ and thus $P(Q/N)$ is a subset of these cosets.

\begin{lemma}\label{LEM:pq-intersects-pqnk}
If $N \normal Q$, $|N| = k$, and $a_1 N, \ldots, a_k N \in P(Q/N)$, then
\[P(Q) \cap (a_1 N \cdots a_k N) \neq \emptyset.\] That is, $P(Q)$ intersects every member of $P(Q/N)^k$.
\end{lemma}
\begin{proof}
Let $|Q| = mk$. For any $aN \in P(Q/N)$, we may select a system of coset representatives of $N$ in $Q$, say $\{x_1, \ldots, x_m\}$, and some association of the left hand side such that
\begin{align}
x_1 N \cdots x_m N = aN \label{EQN:dh-loops-lemma}
\end{align}
and thus using the same association pattern $x_1 \cdots x_m \in aN$. Furthermore, since (\ref{EQN:dh-loops-lemma}) depends only on the order and association of the cosets of $N$ (rather than the specific representatives chosen), we may select $k$ disjoint sets of coset representatives of $N$ in $Q$, say $\{ x_{(i,1)}, \ldots, x_{(i,m)} : 1 \leq i \leq k \}$, and corresponding association patterns such that
\[ (x_{(1,1)} \cdots x_{(1,m)}) \cdots (x_{(k,1)} \cdots x_{(k,m)}) \in a_1 N \cdots a_k N \in P(Q/N)^k\]
for any selection of $a_i N \in P(Q/N)$ for $1 \leq i \leq k$. Having selected each element of $Q$ as a coset representative precisely once, the left-hand side falls in $P(Q)$ and we are done.
\end{proof}

\section{Theorem \ref{THM:kplex-kprod}}\label{SEC:kplex-kprod}

For $x \in Q$, we write $L_x$ ($R_x$) for the left (right) translation of $Q$ by $x$. Our notation for the left translation is not to be confused with the convention of using $L$ for a latin square. The latter usage appears only in our introduction. In any case, the meaning is always made clear by the context.

The \define{multiplication group of $Q$}, written $\Mlt(Q)$, is the subgroup of the symmetric group $S_Q$ generated by all left and right translations, i.e. $\sqn{L_x, R_x : x \in Q}$, while the \define{left multiplication group of $Q$}, written $\LMlt(Q)$, is generated by all left translations. If $H \normal Q$ and $\rho \in \Mlt(Q)$, we may define the map $\rho_H (x H) := \rho(x) H$, which is said to be induced by $\rho$. It is straightforward to verify that the map is well-defined and that $\rho_H \in \Mlt(Q/H)$.

To prove Theorem \ref{THM:kplex-kprod} in the group case one would like to use the fact that from an identity like
\[ a_1 (a_2 (\cdots (a_k x ) \cdots ) = x \]
we may conclude that $a_1 (a_2 (\cdots (a_k ) \cdots ) = 1$, which is trivial in the presence of associativity but typically false otherwise. In the general loop case, the following lemma shows we can at least conclude that
\[ a_1 (a_2 (\cdots (a_k ) \cdots ) \in A(Q). \]

\begin{lemma}\label{LEM:not-derangement}
~\quad
\begin{enumerate}
\item If $\rho \in \LMlt(Q)$, then $\rho_{A(Q)} = L_{\rho(1)A(Q)}$.
\item If $\rho \in \Mlt(Q)$, then $\rho_{Q'} = L_{\rho(1)Q'}$.
\item Since they are left translations, $\rho_{A(Q)}$ and $\rho_{Q'}$ are constant if and only if they have fixed points.
\item If $a_1 (a_2 (\cdots (a_k x ) \cdots ) = x$, then $a_1 (a_2 (\cdots (a_k ) \cdots ) \in A(Q)$.
\end{enumerate}
\end{lemma}
\begin{proof} Set $A := A(Q)$.

(i) Let $\rho = L_{a_1} \cdots L_{a_k}$. Then $\rho_A (qA) = a_1 (a_2 \cdots (a_k q ) \cdots ) A$. Since $Q/A$ is a group, we may reassociate to get $\rho_A (qA) = a_1 (a_2 \cdots (a_k ) \cdots ) A \cdot q A = \rho(1) A \cdot qA$. Thus $\rho_A = L_{\rho(1)A}$.

(ii) Let $\rho = T^{\epsilon_1} _{a_1} \cdots T^{\epsilon_k} _{a_k}$ where $T^{\epsilon_i} \in \{L,R\}$. Since $Q/Q'$ is an Abelian group, we may reassociate and commute to get $\rho_{Q'} (qQ') = a_1 (a_2 \cdots (a_k ) \cdots ) Q' \cdot q Q' =\rho(1) Q' \cdot qQ'$. Thus $\rho_{Q'} = L_{\rho(1)Q'}$.

(iii) Since $\rho_{A(Q)}$ is a left translation, if it has a fixed point, it is constant. Likewise for $\rho_{Q'}$.
 
(iv) Let $\rho (z) := a_1 (a_2 (\cdots (a_k z ) \cdots )$. Since $\rho_{A} (xA) = \rho(x)A = xA$, it is constant by (iii). Thus $\rho_A (A) = A$ and in particular $\rho(1) = a_1 (a_2 (\cdots (a_k ) \cdots ) \in A$.
\end{proof}

Lemma \ref{LEM:not-derangement} is stated somewhat more generally then we actually need. If the translation notation feels cumbersome, the idea is very basic. Given the product $a_1 (a_2 (\cdots (a_k x ) \cdots ) = x$, we may reduce both sides mod $A$ to get
\begin{align*}
a_1 A a_2 A \cdots a_k A x A &= x A \\
a_1 A a_2 A \cdots a_k A &= 1 A \\
a_1 a_2 \cdots a_k &\in A .
\end{align*}

\begin{lemma}\label{LEM:ec-regular}
If $C \neq \emptyset$ is regular, then there exists $C'$ such that 
\begin{enumerate}
\item $\emptyset \neq C' \subseteq C$,
\item $P(C'_r)$ intersects $A(Q)$, and
\item $C \setminus C'$ is regular (and possibly empty).
\end{enumerate}
It follows that $P(C_r)$ intersects $A(Q)$.
\end{lemma}

\begin{proof}
Let $[k] := \{1, \ldots, k\}$. Suppose $C = \{(x_i, y_i, z_i) : i \in [k] \}$ is regular. Select $i_1 \in [k]$ at random. Having selected $i_1, \cdots, i_m \in [k]$, pick $i_{m+1} \in [k]$ such that $y_{i_m} = z_{i_{m+1}}$. Since $C$ is regular, such a selection can always be made. If $i_{m+1} \not \in \{i_1, \ldots, i_m\}$, continue.

Otherwise, stop and consider the set $\{i_j, \ldots, i_m\}$ where $i_j = i_{m+1}$. Reindex $C$ such that $\sqn{i_j, \ldots, i_m} = \sqn{1, \cdots, s}$ and set $C' := \{(x_i, y_i, z_i) : 1 \leq i \leq s\}$. Note that $y_s = z_1$. By construction, $C'$ has the following form:

\begin{align*}
C' = \{
& (x_1, y_1, y_s), \\
& (x_2, y_2, y_1), \\
& (x_3, y_3, y_2), \\
& \cdots \\
& (x_{s-1}, y_{s-1}, y_{s-2}), \\
& (x_{s}, y_{s}, y_{s-1}) \}.
\end{align*}

$C'$ is clearly regular and thus so too is $C \setminus C'$. Furthermore, by construction we have
\begin{align*}
x_1 (x_2 ( \cdots (x_s z_1 ) \cdots )) = z_1.
\end{align*}
By Lemma \ref{LEM:not-derangement}, $x_1 (x_2 ( \cdots (x_s ) \cdots )) \in A(Q)$. Since this product is in $P(C'_r)$ as well, $P(C'_r) \cap A(Q) \neq \emptyset$. Iterating this construction we have $P(C_r)$ intersects $A(Q)$.

\end{proof}

\begin{proof}[Proof of Theorem \ref{THM:kplex-kprod}]
If $C$ is a $k$-plex, then $C_r$ consists of $k$ copies of each element of $Q$ and thus $P(C_r) = P^k(Q)$. By Lemma \ref{LEM:ec-regular}, $P^k(Q)$ intersects $A(Q)$.
\end{proof}

\section{Theorem \ref{THM:one-in-pg-iff-ceregrowtrans}}

\begin{lemma}\label{LEM:col-entry-reg-construction}
If $G$ is a group and $g_1, \ldots, g_k \in G$ such that $g_1 \cdots g_k = 1$ and no proper contiguous subsequence evaluates to $1$, then $G$ admits a regular set $C$ such that $C_r = \{g_1, \ldots, g_k\}$ and no column (and thus no entry) is selected more than once.
\end{lemma}
\begin{proof}
Set $h_i := g_{i+1} \cdots g_k$ for $1 \leq i \leq k-1$ and $h_0 = h_k := 1$. Note that we have $g_i h_i = h_{i-1}$ for $1 \leq i \leq k$. We claim that $C := \{(g_i,h_i,h_{i-1}) : 1 \leq i \leq k\}$ is the desired regular set. It is clear that $C$ is regular and that $C_r = \{g_1, \ldots, g_k\}$. To see that no column is selected more than once, suppose that $h_i = h_{i+j}$ for $j\geq1$. That is, $g_{i+1} \cdots g_k = g_{i+j+1} \cdots g_k$. Canceling on the right, we have $g_{i+1} \cdots g_{i+j} = 1$, a contradiction.
\end{proof}

\begin{proof}[Proof of Theorem \ref{THM:one-in-pg-iff-ceregrowtrans}] ~

(i) $\Longrightarrow$ (ii)
In this case we may use the standard argument from the group case showing that a single transversal extends to $n$ disjoint transversals. Let $T = \{(x_i, y_i, z_i) : 1 \leq i \leq n \}$ be a regular row transversal of $G$. For each $g \in G$, form $T_g := \{(x,yg, zg) : (x,y,z) \in T\}$. It is easy to check that the family $\{T_g : g \in G\}$ partitions the multiplication table of $G$ into regular row transversals.

(i) $\Longleftarrow$ (ii)
If $G$ admits a partition into regular row transversals, then it certainly has a regular row transversal.

(i) $\Longrightarrow$ (iv)
Let $T$ be a regular row transversal. By Theorem \ref{THM:kplex-kprod}, $P(G) \cap A(G) \neq \emptyset$ and thus $1 \in P(G)$.

(i) $\Longleftarrow$ (iv)
Let $g_1 \cdots g_n = 1$. We partition $G$ as follows:
\begin{itemize}
\item If no proper contiguous subsequence of $g_1 \cdots g_n$ evaluates to $1$, stop.
\item Otherwise, extract the offending subsequence $g_i \cdots g_j = 1$ and note that
\[g_1 \cdots g_{i-1} g_{j+1} \cdots g_n = 1.\]
\item Iterate this process with these shortened products.
\end{itemize}

Suppose we have thus partitioned $G$ into $k$ disjoint sequences $\{g_{(i,1)}, \ldots, g_{(i,n_i) : 1 \leq i \leq k}\}$ such that $g_{(i,1)} \cdots g_{(i,n_i)} = 1$ for $1 \leq i \leq k$ and no proper contiguous subsequence of $g_{(i,1)}, \ldots, g_{(i,n_i)}$ evaluates to $1$. Now we apply Lemma \ref{LEM:col-entry-reg-construction} to each subsequence to get regular sets $C_i$ for $1 \leq i \leq k$. Then $\bigcup_{i=1} ^k C_i$ is a regular row transversal of $G$.

(iii) $\iff$ (iv)
As noted earlier, this is an established equivalence in the Hall-Paige conjecture.

\end{proof}

\section{Theorem \ref{THM:pqnot-qsteptype}}\label{SEC:qstep-kprod-equiv}

\begin{lemma}\label{LEM:hp-abelian}
$(2)$ $\iff$ $(3)$ holds for Abelian groups.
\end{lemma}
\begin{proof}
As mentioned above, Vaughan-Lee and Wanless give a direct, elementary proof of this result for all groups \cite{Vaughan-Lee:2003cr}. For an earlier though indirect proof, Paige showed that (1) $\iff$ (2) holds in Abelian groups \cite{Paige:1947jt} and Hall and Paige showed that (1) $\iff$ (3) in solvable groups \cite{Hall:1955zr}.
\end{proof}

\begin{lemma}\label{LEM:burnside}
If a group $G$ has a cyclic Sylow $2$-subgroup $S$, then there exists $N \normal G$ such that $G/N \cong S$.
\end{lemma}
\begin{proof}
This is a direct application of Burnside's Normal Complement theorem that can be found in most graduate level group theory texts (see \cite{Zassenhaus:1949ek} for example).
\end{proof}

\begin{proof}[Proof of Theorem \ref{THM:pqnot-qsteptype}]
($\Longleftarrow$) We show the contrapositive. Suppose $P(Q) \subseteq aQ' \neq Q'$. Since $G := Q/Q'$ is an Abelian group, $P(G) = \{bQ'\}$ such that $b^2 \in Q'$. By Lemma \ref{LEM:pq-intersects-pqnk}, $P(Q)$ intersects every element of $P(G)^{|Q'|} = \{b^{|Q'|} Q'\}$ and thus $aQ' = b^{|Q'|} Q'$. Since $aQ' \neq Q'$ and $b^2 \in Q'$, it follows that $aQ' = bQ'$ and $|Q'|$ is odd.

Since $P(G) \neq \{1Q'\}$, by Lemmas \ref{LEM:hp-abelian} and \ref{LEM:burnside} there is $N \normal G$ such that $|N|$ is odd and $G/N \cong \ZZ_{2^m}$. $N$ is a collection of coset of $Q'$. Letting $H$ be their union, we have $Q/H \cong G/N \cong \ZZ_{2^m}$ and $|H| = |N||Q'|$ is odd.

($\Longrightarrow$) Again we argue the contrapositive.
Suppose $N \normal Q$ such that $|N| = q$ is odd and $Q/N \cong \ZZ_{2^m}$ for $m \geq 1$. Since $Q/N \cong \ZZ_{2^m}$, $P(Q/N) = \{aN\} \neq \{N\}$ such that $a^2 \in N$. By Lemma \ref{LEM:pq-intersects-pqnk}, $P(Q)$ intersects every element of $P(Q/N)^{|N|} = \{aN\}^{|N|} = \{aN\}$.

Given that $Q/N$ is an Abelian group, $Q' \subseteq N$ but since $P(Q)$ intersects $aN \neq N$, it is therefore disjoint from $Q'$.
\end{proof}

\section{Theorem \ref{THM:DH-loops}}\label{SEC:DenesHermann}

An \define{$A$-loop} is a loop in which all inner mappings are automorphisms. The variety of $A$-loops is larger than that of groups but is certainly not all loops. Bruck and Paige conducted the earliest extensive study of $A$-loops \cite{Bruck:1956mz}.

Before proving Theorem \ref{THM:DH-loops}, we make several additional observations about the sets $P^k(Q)$. While none of these results will be used directly in our proof, we hope they are of some interest in that they may suggest an alternative proof of the D\'enes-Hermann theorem.

\begin{lemma}\label{LEM:pomega}
Set $\PI := \bigcup _{i=1} ^{\infty} P^i (Q)$.
\begin{enumerate}
\item $\PI \leq Q$,
\item $\PI = P^k(Q) \cup P^{k+1}(Q)$ for sufficiently large $k$,
\item If $\PI \normal Q$, then $\PI = Q'$ or $\PI = Q' \cup aQ'$ where $a^2 \in Q'$, and
\item $\PI$ is fixed by all automorphisms of $Q$. Thus, if $Q$ is an $A$-loop, then $\PI \normal Q$.
\end{enumerate}
\end{lemma}
\begin{proof}
(i) Since $Q$ is finite, we need only verify that $\PI $ is closed under multiplication. If $x,y \in \PI $, then $x \in P^i(Q)$ and $y \in P^j(Q)$ for some $i,j \geq 1$. Thus $xy \in P^{i+j}(Q) \subseteq \PI $. 

(ii) Again, since $Q$ is finite, the nested sequence $(P^{2i} (Q): 1 \leq i < \infty)$ must terminate at some step, say $P^{k_1}(Q)$. Likewise $(P^{2i + 1} (Q): 1 \leq i < \infty)$ must terminate at some step, say $P^{k_2}(Q)$. Thus letting $k = \max \{k_1, k_2\}$, we have $\PI  = P^k(Q) \cup P^{k+1}(Q)$. (In fact, by Lemma \ref{LEM:pq-properties} part (ii), the sequences terminate at the same time.)

(iii) Suppose $\PI  \normal Q$. We show that $Q / \PI $ is an Abelian group and thus $Q' \subseteq \PI $. To see that $Q/\PI $ is a group, note that $a\PI  b\PI  \cdot c\PI  = (ab \cdot c) \PI $. We would like to show that $(ab \cdot c) \PI  = (a \cdot bc) \PI $.

To that end, let $a' \in P(Q \setminus \{a\})$ and likewise for $b'$ and $c'$.  We translate both $(ab \cdot c) \PI $ and  $(a \cdot bc) \PI $ by $(a' b' \cdot c') \PI $ on the left to get
\begin{align*}
(a \cdot bc) \PI  \cdot (a' b' \cdot c') \PI  &= [(a \cdot bc) \cdot (a' b' \cdot c')] \PI  \\
(ab \cdot c) \PI  \cdot (a' b' \cdot c') \PI  &= [(ab \cdot c) \cdot (a' b' \cdot c')] \PI 
\end{align*}
Note that both $(a \cdot bc) \cdot (a' b' \cdot c')$ and $(ab \cdot c) \cdot (a' b' \cdot c')$ are elements of $P(Q)$ and thus both right-hand sides reduce to $\PI $. Thus both $(ab \cdot c) \PI $ and $(a \cdot bc) \PI $ are left inverses of $(a' b' \cdot c')\PI $. Since left inverses are unique, we have $(ab \cdot c) \PI  = (a \cdot bc) \PI $ and $Q/\PI $ is a group.

To see that $Q/\PI $ is Abelian, consider $a\PI  b\PI $ and $b\PI  a\PI $. Again let $a' \in P (Q \setminus \{a\})$ and $b' \in P (Q \setminus \{b\})$. We then have
\begin{align*}
ab \PI  \cdot a'b' \PI  &= (ab \cdot a'b') \PI  \\
ba \PI  \cdot a'b' \PI  &= (ba \cdot a'b') \PI 
\end{align*}
Since $(ab \cdot a'b')$ and $(ba \cdot a'b')$ are both members of $P(Q)$, the right-hand sides reduce to $\PI $. As above, it follows that $a\PI  b\PI  = b\PI  a\PI $.

Since $Q'$ is the smallest normal subloop of $Q$ such that $Q / Q'$ is an Abelian group, $Q' \subseteq \PI $.

(iv) First note that for $i \geq 1$, $P^i(Q)$ is always fixed by automorphisms of $Q$ and thus so is $\PI $. If $Q$ is an $A$-loop, then $\PI $ is fixed by every inner-mapping. That is, $\PI  \normal Q$.
\end{proof}

In the spirit of the observations made in Lemma \ref{LEM:pomega}, Yff showed that when $G$ is a group, $P^3(G)$ coincides with a coset of $G'$ \cite[p. 269]{Yff:1991dq}. Although this fact is an easy application of the D\'enes-Hermann theorem, his proof applies directly to all finite groups and avoids the use of the Feit-Thompson theorem.

To our knowledge, Theorem \ref{THM:DH-loops} is the first extension beyond groups of the D\'enes-Hermann theorem. Although our generalization is rather modest, we suspect the result extends almost completely.

\begin{question}\label{Question:Denes-Hermann-All}
If $|Q|$ is sufficiently large, does it follow that $P(Q)$ is a coset of $Q'$?
\end{question}

The D\'enes-Hermann theorem is equivalent to the claim that for any finite group $G$
\[
|\{g_1 \cdots g_n : \mbox{ ranging over all orderings }\}| = |G'|.
\]
To answer Question \ref{Question:Denes-Hermann-All} affirmatively, it would suffice to show the perhaps stronger claim that given any fixed ordering of the elements of $Q$, we have
\[
|\{q_1 \cdots q_n : \mbox{ ranging over all associations } \}| = |A(Q)|.
\]

The restriction on $|Q|$ in Question \ref{Question:Denes-Hermann-All} is necessary as Vojt\u{e}chovsk\'y and Wanless have observed that for $3$ of the $5$ non-associative loops of order $5$, $P(Q)$ has order $4$ whereas $Q = Q'$ (see Figure \ref{FIG:DH-counterexample} for one example) \cite{Vojtechovsky:2009jt}. At the time of this writing, we are unaware of any other loops with this property.

\begin{figure}[htbp]
\begin{center}
\begin{tabular}{c|ccccc} 
$Q$ & 1 & 2 & 3 & 4 & 5 \\ \hline 
1 & 1 & 2 & 3 & 4 & 5 \\  
2 & 2 & 1 & 4 & 5 & 3 \\ 
3 & 3 & 5 & 1 & 2 & 4 \\ 
4 & 4 & 3 & 5 & 1 & 2 \\ 
5 & 5 & 4 & 2 & 3 & 1 \\ 
\end{tabular}
\caption{Loop $Q$ for which $Q = Q'$ but $P(Q) = Q \setminus \{1\}$. As such, $Q$ demonstrates the necessity of the cardinality restriction in Question \ref{Question:Denes-Hermann-All}.}
\label{FIG:DH-counterexample}
\end{center}
\end{figure}

We recall the following special case of the correspondence and isomorphism theorems, proofs of which can be found in most standard universal algebra texts.

\begin{lemma}
If $N \normal Q$ and $N \leq H \leq Q$, then 
\begin{enumerate}
\item $H \normal Q$ if and only if $H/N \normal Q/N$ and
\item when $H \normal Q$, $Q/H \cong (Q/N) / (H/N)$.
\end{enumerate}
\end{lemma}

We employ the following lemma in our proof of Theorem \ref{THM:DH-loops}.

\begin{lemma}\label{LEM:derived-fraction}
$(Q / A(Q))' = Q' / A(Q)$.
\end{lemma}

\begin{proof}
Set $A := A(Q)$.
By definition, $(Q / A)'$ is the smallest normal subloop of $Q/A$ such that the factor loop is an Abelian group. Since $A \leq Q' \normal Q$, by the correspondence theorem we have $(Q'/A) \normal (Q/A)$ and $(Q/A) / (Q'/A) \cong Q/Q'$, an Abelian group. Thus $(Q/A)' \leq (Q'/A)$.

We now show $(Q'/A) \leq (Q/A)'$. Fix $N/A \normal Q/A$ such that $(Q/A)/(N/A)$ is an Abelian group. Again by the correspondence theorem, $N \normal Q$ and $Q/N \cong (Q/A) / (N/A)$. Since $Q/N$ is an Abelian group, $Q' \leq N$ and thus $Q' / A \leq N / A$. It follows that $(Q'/A) \leq (Q/A)'$

\end{proof}

\begin{proof}[Proof of Theorem \ref{THM:DH-loops}]
Let $A := A(Q)$ and $k := |A|$. Since $P(Q)$ is contained in a single coset of $Q'$, it suffices to show that $P(Q)$ intersects at least $[Q':A]$ cosets of $A$ (the maximum possible).

By Theorem \ref{THM:denes-hermann}, $P(Q/A) = \{ xA (Q/A)' \}$ such that $x^2 A \in (Q/A)'$. By Lemma \ref{LEM:derived-fraction}, $xA (Q/A)' = xA (Q' / A) = (xQ') A$. Thus we have $P(Q/A) = \{ (xQ') A \}$. By Lemma \ref{LEM:pq-intersects-pqnk}, $P(Q)$ intersects each of the $[Q' : A]$ elements of $P(Q/A)^k = \{ (x^k Q') A \} = \{ q A : q \in x^k Q' \} $.
\end{proof}

\section{Concluding Remarks}

We have proposed the HP-condition as a possible framework for extensions of Theorem \ref{THM:HP} from groups into the larger world of non-associative loops. Having shown several universal implications between the points of the HP-condition, we leave open the difficult problem of identifying interesting varieties of loops in which conditions (B) and (C) imply (A).

It would also be of interest to identify classes of loops in which the existence of a regular row transversal implies the existence of a transversal. As noted in Corollary \ref{COR:HP-equiv-reg-row-trans}, this implication in groups is fully equivalent to Theorem \ref{THM:HP}.

I thank Michael Kinyon and Petr Vojt\v{e}chovsk\'y for their helpful conversations regarding this material and Anthony Evans and Ian Wanless for sharing several articles and preprints related to the Hall-Paige conjecture. Lastly I thank the anonymous referee whose feedback clarified the historical background and helped to focus the exposition. 

\bibliographystyle{abbrv}

\begin{thebibliography}{10}

\bibitem{Bray:2008yg}
J.~Bray.
\newblock Unpublished notes.
\newblock 2008.

\bibitem{Bruck:1958ly}
R.~H. Bruck.
\newblock {\em A survey of binary systems}.
\newblock Ergebnisse der Mathematik und ihrer Grenzgebiete. Neue Folge, Heft
  20. Reihe: Gruppentheorie. Springer Verlag, Berlin, 1958.

\bibitem{Bruck:1956mz}
R.~H. Bruck and L.~J. Paige.
\newblock Loops whose inner mappings are automorphisms.
\newblock {\em Ann. of Math. (2)}, 63:308--323, 1956.

\bibitem{Bryant:2009zl}
D.~Bryant, J.~Egan, B.~Maenhaut, and I.~M. Wanless.
\newblock Indivisible plexes in latin squares.
\newblock {\em Designs, Codes and Cryptography}, 52(1), July 2009.

\bibitem{Colbourn:1996cj}
C.~J. Colbourn and J.~H. Dinitz.
\newblock {\em The CRC Handbook of Combinatorial Designs}.
\newblock CRC Press, Boca Raton, FL, 1996.

\bibitem{Denes:1982qf}
J.~D{\'e}nes and P.~Hermann.
\newblock On the product of all elements in a finite group.
\newblock In {\em Algebraic and geometric combinatorics}, volume~65 of {\em
  North-Holland Math. Stud.}, pages 105--109. North-Holland, Amsterdam, 1982.

\bibitem{Denes:1974ul}
J.~D{\'e}nes and A.~D. Keedwell.
\newblock {\em Latin squares and their applications}.
\newblock Academic Press, New York, 1974.

\bibitem{Denes:1989qp}
J.~D{\'e}nes and A.~D. Keedwell.
\newblock A new conjecture concerning admissibility of groups.
\newblock {\em European J. Combin.}, 10(2):171--174, 1989.

\bibitem{Denes:1991lq}
J.~D{\'e}nes and A.~D. Keedwell.
\newblock {\em Latin squares}, volume~46 of {\em Annals of Discrete
  Mathematics}.
\newblock North-Holland Publishing Co., Amsterdam, 1991.

\bibitem{Egan:2008eu}
J.~Egan and I.~M. Wanless.
\newblock Latin squares with no small odd plexes.
\newblock {\em Journal of Combinatorial Designs}, 16(6):477--492, 2008.

\bibitem{Evans:1992zl}
A.~B. Evans.
\newblock The existence of complete mappings of finite groups.
\newblock In {\em Proceedings of the Twenty-third Southeastern International
  Conference on Combinatorics, Graph Theory, and Computing (Boca Raton, FL,
  1992)}, volume~90, pages 65--75, 1992.

\bibitem{Evans:2008gy}
A.~B. Evans.
\newblock The admissibility of sporadic simple groups.
\newblock {\em J. Algebra}, 321:1407--1428, 2009.

\bibitem{Hall:1955zr}
M.~Hall and L.~J. Paige.
\newblock Complete mappings of finite groups.
\newblock {\em Pacific J. Math.}, 5:541--549, 1955.

\bibitem{Paige:1947jt}
L.~J. Paige.
\newblock A note on finite {A}belian groups.
\newblock {\em Bull. Amer. Math. Soc.}, 53:590--593, 1947.

\bibitem{Paige:1951hb}
L.~J. Paige.
\newblock Complete mappings of finite groups.
\newblock {\em Pacific J. Math.}, 1:111--116, 1951.

\bibitem{Pflugfelder:1990bh}
H.~O. Pflugfelder.
\newblock {\em Quasigroups and loops: introduction}, volume~7 of {\em Sigma
  Series in Pure Mathematics}.
\newblock Heldermann Verlag, Berlin, 1990.

\bibitem{Pula:2008xw}
K.~Pula.
\newblock The {H}all-{P}aige conjecture for finite {M}oufang loops.
\newblock In preparation, 2009.

\bibitem{Vaughan-Lee:2003cr}
M.~Vaughan-Lee and I.~M. Wanless.
\newblock Latin squares and the {H}all-{P}aige conjecture.
\newblock {\em Bull. London Math. Soc.}, 35(2):191--195, 2003.

\bibitem{Vojtechovsky:2009jt}
Vojt\v{e}chovsk\'y and Wanless.
\newblock Private correspondence.
\newblock February 2009.

\bibitem{Wanless:2002rr}
I.~M. Wanless.
\newblock A generalisation of transversals for {L}atin squares.
\newblock {\em Electron. J. Combin.}, 9(1):Research Paper 12, 15 pp.
  (electronic), 2002.

\bibitem{Wanless:2007zr}
I.~M. Wanless.
\newblock Transversals in {L}atin squares.
\newblock {\em Quasigroups Related Systems}, 15(1):169--190, 2007.

\bibitem{Wilcox:2007qf}
S.~Wilcox.
\newblock Reduction of the hall-paige conjecture to sporadic simple groups.
\newblock {\em J. Algebra}, 321(5):1407--1428, March 2009.

\bibitem{Yff:1991dq}
P.~Yff.
\newblock On the {D}\'enes-{H}ermann theorem: a different approach.
\newblock {\em European J. Combin.}, 12(3):267--270, 1991.

\bibitem{Zassenhaus:1949ek}
H.~Zassenhaus.
\newblock {\em The theory of groups}.
\newblock Chelsea Publishing Co., New York, New York, 1949.

\end{thebibliography}

\end{document}